%% file: EA.tex
\documentclass[11pt, reqno]{amsart}

\usepackage[letterpaper, hmargin=1.0in, vmargin=1.0in]{geometry}
\usepackage{caption}

\usepackage[scale=1]{inconsolata}
\usepackage{bm}

\def\es{{e^\star}}
\def\cond{\ \Big| \ }
\renewcommand{\leq}{\leqslant}
\renewcommand{\geq}{\geqslant}

\def\m{\mathbf{m}}

\input{preamble.tex}

\input{additional_macros.tex}
\usetikzlibrary{patterns}
\renewcommand{\k}{\mathbf{k}}

\let\temp\phi
\let\phi\varphi
\let\varphi\temp

\begin{document}
\title[Zero temperature Edwards-Anderson]{Spectral properties of the zero temperature Edwards-Anderson model}

 \author{Mriganka Basu Roy Chowdhury, Shirshendu Ganguly }
 \address{Mriganka Basu Roy Chowdhury\\ University of California, Berkeley}
 \email{mriganka\_brc@berkeley.edu}
 \address{Shirshendu Ganguly\\ University of California, Berkeley}
 \email{sganguly@berkeley.edu}

\begin{abstract} An Ising model with random couplings on a graph is a model of a spin glass. While the mean field case of the Sherrington-Kirkpatrick model is very well studied, the more realistic lattice setting, known as the Edwards-Anderson (EA) model, has witnessed rather limited progress. In \cite{chatterjee2023spin} chaotic properties of the ground state in the EA model were established via the study of the Fourier spectrum of the two-point spin correlation. A natural direction of research concerns fractal properties of the Fourier spectrum in analogy with critical percolation. In particular, numerical findings \cite{bm2} seem to support the belief that the fractal dimension of the associated spectral sample drawn according to the Fourier spectrum is strictly bigger than one. Towards this, in this note we introduce a percolation-type argument, relying on the construction of ``barriers'', to obtain new probabilistic lower bounds on the size of the spectral sample. 
\end{abstract}
\maketitle{}

\setcounter{tocdepth}{1}
\tableofcontents

\parindent=0pt
\parskip=5pt

\section{Introduction}
Consider the following spin model on a general connected graph $G = (V, E)$. 
Each edge
$e \in E$ is assigned a standard Gaussian variable $J_e \sim \Nor{0, 1}$ in an i.i.d. fashion with these variables
referred to as \emph{couplings}, and given the disorder, consider the (random) Hamiltonian on \emph{spin configurations} 
$\sig \in \{-1, +1\}^V$ given by
\begin{align}\label{hamil23}
	H(\sig) = \sum_{e  = \{u, v\}\in E} J_e \sig_{u} \sig_{v}.
\end{align}
The object of interest is then the probability measure on spin configurations given by $\mu_\be(\sig) = \f{1}{Z}
\Exp{\be H(\sig)}$, where $Z$ is the (random) partition function and $\be$ is the inverse temperature. 
When the couplings $J_e$ are taken to be a constant, say $1$, then this corresponds to the well known Ising model. Thus the above Gibbs measure is a variant of the Ising model with random, mean-zero couplings, expected to exhibit glassy-behavior and thus serves as a toy model for spin-glasses. 
The popular
Sherrington-Kirkpatrick (SK) model, see e.g., \cite{sk,pan}, is a special case of the above, wherein the underlying graph is the complete graph.

Note that the usual convention is to take the Gibbs measure to be the exponential of the \emph{negative} of the 
Hamiltonian to ensure that the Gibbs measure seeks out \emph{low-$H$ states} (low
energy), but we will work with the current formulation without the negative sign for expository reasons. However, in our setting, since the Gaussians are symmetric, this can be done without loss of generality.

While the SK model has been studied in great detail, the more realistic lattice setting (known primarily as the Edwards-Anderson (EA) model introduced in \cite{ea}) is significantly more challenging. Most of the questions for the EA model remain open notwithstanding a significant amount of non-rigorous literature devoted to predictions, often without consensus, of possible behaviors. See e.g., \cite{bm1}. A more mathematical treatment may be found in \cite{ns}.  The EA model will be the focus of  this note.  The primary aim is to introduce a percolation-type ``barrier'' argument to understand the Fourier spectrum of the two-point function associated to \emph{ground states}, i.e., configurations $\sig$ that maximize
$H(\sig)$. The Fourier basis in the setting where the disorder is given by i.i.d. Rademacher variables is in bijection with the subsets of edges, and the Gaussian case, which is the setting of the paper, is, albeit more complicated, not very different. While the results presented in the paper are obtained by initial applications of the idea, we hope that it has further purchase and can lead to new bounds in the study of lattice spin glass models where mathematical progress has been few and far between. 

With the above background, let us now present the setup of our results. 

We will use $V = \{(i, j) \in \Z^2: 1 \leq i, j \leq n\}$ to denote the vertex set of the  $n \times n$ subgrid of $\Z^2$ endowed with  nearest neighbor
edges (our arguments are not planar in nature and will hold in any dimension as will be remarked later but for concreteness we fix the dimension to be $2$ for the moment). It is easily seen that the ground state $\sig = \argmax_{\sig} H(\sig)$ is unique with probability 1, on account
of the continuity of Gaussians, up to a global sign flip (that is, $\pm \sig$ are both ground states). Thus, to understand correlations, one must consider the relative spin for two vertices $u, v \in V$, i.e., the product $\sig_u \sig_v$.

Observe that $\sig_u \sig_v$ is simply a function of the couplings $\{J_e\}$.  The particular perspective we will explore is a dynamic one through the lens of noise-sensitivity. That is, how sensitively does the relative spin between two vertices react to perturbations of the underlying disorder. This is the subject of disorder chaos which was introduced in the context of spin glasses in \cite{fh,bm2}. While for the mean field SK model, there has been impressive progress \cite{sc,chachaos, auf,chen1,chen2,chen3,eldan}, a rigorous study of this in the EA setting was initiated in \cite{chatterjee2023spin}. Related results may also be found in \cite{arguin,ans}. While there are many natural perturbation models, a particularly canonical one when the underlying disorder variables are i.i.d. standard Gaussians is to evolve each of them along  independent stationary Ornstein-Uhlenbeck (OU) flows. 
The OU process admits the following well known SDE representation: 
\begin{align*}
d X_t=-X_tdt+\sqrt{2}dB_t. 
\end{align*}
We will evolve each coupling $J_e$ independently according to 
an OU $J_e^t$ ($t \geq 0$) such that $J_e^0 = J_e$. 
In particular, for any fixed $t>0$ the following representation holds: $J_e^t=e^{-t}J_e^0+\sqrt{1-e^{-2t}}J'_e$ where $J_e'$ is an independent standard Gaussian.  Let the ground state at time $t$ under the OU flow be denoted by $\sigma^t$.

Finally, as is common in investigations of noise-sensitivity properties, a rather fruitful technique is to consider the
\emph{Fourier expansion} of $\sig_u\sig_v$, as a function of the couplings, which we now turn to. Relaxation properties of the relative spin can then be used in the study of chaotic properties of the ground state through its site overlap, i.e., $R(t)=\frac{1}{|V|}\langle\sigma^0, \sigma^{t}\rangle.$  Since the right hand side is well defined only up to a sign, it is natural to consider $R(t)^2 \in [0,1],$ which can be expanded into a sum of relative spin terms as in \eqref{expan1} below. In this formulation, it is of interest to find the smallest scale of $t$ at which $R(t)$ becomes $o(1),$ i.e. when $\sigma^0$ and $\sigma^t$ are nearly orthogonal to each
other.

Since our couplings are Gaussian, the appropriate Fourier basis is provided by the \emph{Hermite polynomials}. For
one Gaussian variable, say $\xi$, the Hermite polynomials are one-variable polynomials $h_k(x)$ of degree $k$ ($k \geq
0$), satisfying the orthonormality relations
\begin{align*}
	&\E h_k(\xi)h_m(\xi) = 0, \quad k \neq m, \\
	&\E h_k^2(\xi) = 1, \quad k \geq 0.
\end{align*}
The first few Hermite polynomials are given by $h_0(x) = 1, h_1(x) = x, h_2(x) = 2^{-1/2}\Rnd{x^2 - 1}, h_3(x)
= 6^{-1/2}\Rnd{x^3 - 3x}$. One may decompose any ($L^2$) function $f$ of one Gaussian variable $\xi$ as 
\begin{align*}
	f(\xi) = \sum_{k=0}^\infty \al_k h_k(\xi), \quad \al_k = \E\Box{f(\xi) h_k(\xi)}.
\end{align*}

In our setup, we require an analogous expansion for functions of \emph{multiple} Gaussians $\{J_e\}$, which may be
simply obtained by taking tensor products of Hermite polynomials. For each \emph{multi-index} $\k = (k_e)_{e
\in E} \in \N^E$ (we will use the notation $\N = \{0, 1, 2, \ldots\}$), define the corresponding basis function
\begin{align*}
	h_\k(\{J_e\}) = \prod_{e \in E} h_{k_e}(J_e).
\end{align*}
Given the above, we may Fourier-expand $\sig_u\sig_v$ as follows:
\begin{align}\label{fourierexp}
	\sig_u \sig_v = \sum_{\k \in \N^E} \al_\k h_\k(\{J_e\}), \quad \text{where}\quad \al_\k = \E\Box{\sig_u \sig_v h_\k(\{J_e\})}.
\end{align}

That $\alpha_0=\E\Box{\sig_u^0 \sig_v^0}=0$ follows from the fact that the annealed distribution (averaged over the Gaussians) of $\sigma$ is uniform on $\{-1,1\}^V$ up to a global sign flip. One way to see this is via the observation that for any spin vector $\tau\in\{-1,1\}^V,$ the mapping $J_{(u,v)} \to J^\tau_{(u,v)}\coloneq \tau_{u} J_{(u,v)} \tau_v$ is a measure preserving bijection between the Gaussian variables $\{J_{e}\}_e$ and $\{J^\tau_{e}\}_e.$ Now this map keeps the quadratic forms $H(\sigma)$ from \eqref{hamil23} invariant. Namely, $H(\sigma)=H^\tau(\sigma^{\tau})$ where $\sigma^{\tau}=\tau \sigma$ with the latter being the pointwise product and $H^{\tau}$ is the quadratic form where $\{J_{e}\}_e$ in \eqref{hamil23}  is replaced by $\{J^\tau_{e}\}_e$ (this is because, owing to the fact that $\tau_w^2=1$ for any $w,$ for each edge $e=(u,v)$, $\sigma_u J_{e}\sigma_{v}=\sigma_u\tau_{u} \tau_u J_{e}\tau_v \tau_v\sigma_{v}=\sigma^{\tau}_u J^\tau_e\sigma^\tau_v$).  Thus the ground state corresponding $\{J^{\tau}_{e}\}_e$ is obtained by multiplying pointwise by $\tau$ the ground state corresponding to $\{J_{e}\}_e.$ This proves that the annealed ground state distribution is invariant under pointwise multiplication by any $\tau \in \{-1,1\}^V$ and hence must be uniform.

Recalling that ground state at time $t$ under the OU flow is denoted by $\sigma^t$, using \eqref{fourierexp},
we now arrive at the key observable, the
covariance between $\sig_u^t \sig_v^t$ and $\sig_u^0 \sig_v^0$:
\begin{align}\label{expan1}
	\Cov(\sig_u^t \sig_v^t, \sig_u^0 \sig_v^0) &= \E \Box{\sig_u^t \sig_v^t \sig_u^0 \sig_v^0} -
	\underbrace{\E\Box{\sig_u^t \sig_v^t} \E\Box{\sig_u^0 \sig_v^0}}_{= 0} \\
											   &= \sum_{\k, \m \in \N^E} \al_\k \al_\m \cdot \E[h_\k(\{J_e^t\}) h_\m(\{J_e^0\})] \\
											   &= \sum_{\k, \m \in \N^E} \al_\k \al_\m \cdot \prod_{e \in E} \E[h_{k_e}(J_e^t) h_{m_e}(J_e^0)].
\end{align}
This is the moment where the usefulness of the Fourier approach becomes apparent. 
It is a fact that for any two Hermite polynomials $h_k, h_m$, the covariance $\E[h_k(J_e^0) h_m(J_e^t)]$ is zero if $k
\neq m$, and is $e^{-kt}$ if $k = m$. {This follows from the fact that the Hermite polynomials $h_k$ are
eigenfunctions of eigenvalues $-k$ of the generator of the OU process, see for instance \cite{sc}}. Consequently, 
\begin{align*}
	\Cov(\sig_u^t \sig_v^t, \sig_u^0 \sig_v^0) &= \sum_{\k \in \N^E} \al_\k^2 \prod_{e \in E} e^{-k_e t} = \sum_{\k \in
	\N^E} \al_\k^2 e^{-\Abs{\k} t},
\end{align*}
where we define $\Abs{\k} = \sum_{e \in E} k_e$, to be referred to as the \emph{weight} of $\k$. Therefore, if the Fourier expansion of
$\sig_u\sig_v$ is mostly supported on large weight terms, the decorrelation will be faster. Note that since
$\sig_u\sig_v = \pm 1$, the result above for $t = 0$ is simply the Parseval's Fourier isometry result, i.e.,
\begin{align*}
	1 = \Var(\sig_u\sig_v) = \sum_{\k \in \N^E} \al_\k^2,
\end{align*}
The above allows an interpretation of the Fourier coefficients as a probability mass function.

\begin{definition}\label{specssample} The \emph{spectral (probability) measure} $\mu$ on $\N^E$, a sample from which is termed as the spectral sample, is given by \begin{align}
	\label{eq:mudef}
	\mu(\k) =  \al_\k^2.
\end{align}
In this article, instead of integer vectors, we will deal with the more geometric object, their supports. For any such $\k,$ let $E_\k = \{e \in E: k_e \geq 1\}$ be the \emph{support} of $\k$.
With the above definition, for any $S \sse E,$ define $$\mu(S)=\sum_{\k: E_\k = S}\mu(\k).$$
\end{definition}

In the above terms, note that 
\begin{align*}
	\sum_{\k \in \N^E} \al_\k^2 e^{-\Abs{\k} t} = \E_{\k \sim \mu} \Box{e^{-\Abs{\k} t}},
\end{align*}
Thus a large size of the typical spectral sample implies faster de-correlation properties.

A seminal success story in the study of noise sensitivity is the investigation of spectral properties of crossing events in critical percolation carried out in \cite{bks, gps}. A particularly important outcome was a rather detailed understanding of the corresponding spectral sample
and its fractal properties. 
Drawing a parallel, one may speculate that even in the EA model, the spectral sample continues to exhibit fractal behavior. One particular manifestation of it could be its size.

Given the above preparation and background, we are now in a position to state our main results.

\section{Main results}
The starting point of our results is the following recent result established in \cite{chatterjee2023spin} relying on an  observation regarding the symmetry of the problem, which we will review
shortly. 
\begin{proposition}\cite[Lemma 1]{chatterjee2023spin}
	\label{thm:chatterjee}
	 If $E_\k$ does not connect $x$ and $y$, then
	$\al_\k = 0$.
\end{proposition}

Thus, the spectral sample, with probability one, connects $x$ and $y$, rendering the analogy with percolation more apparent. In particular, the spectral sample is guaranteed to have size at least $d=d(x,y)$ where the latter is the $L^1$ distance between $x,y.$ The general question of a geometric flavor we want to answer is whether there exists $\kappa>0$ such that the typical size of the spectral sample is of the order of $d^{1+\kappa}.$ As already alluded to, numerical findings in \cite{bm2} seem to support such a statement.

We are far from answering this question at the moment; nonetheless, we present the following results as some initial progress in this direction. While Proposition \ref{thm:chatterjee} states that a hard connectivity constraint is necessary to have non-zero spectral mass, our statements are more probabilistically quantitative in nature.

We first state the most vanilla form of the result followed by a stronger form. 

Let  $u = (1, n/2), v = (n, n/2)$ where we assume that $n$ is even. Our first result says that the line $\cL$ connecting $u$ and $v$ has exponentially small mass under the spectral measure $\mu$ defined in Definition \ref{specssample}.

\begin{theorem}
\label{thm:main} There exists $c>0$ such that the following holds. 
	Consider the $n \times n$ grid $V$ with nearest neighbor connections, and consider the vertices $u = (1, n/2), v = (n,
	n/2)$, such that $\norm{u - v} = n-1$ (here $\norm{\cdot}$ is the Manhattan or $L^1$ distance). Consider
	the horizontal line $\cL = \{\big((k, n/2),(k+1, n/2)\big) : 1 \leq k < n\}$ connecting $u$ and $v$ (this is a collection of edges rather than vertices, since spectral samples are subsets of edges). Then
\[
\mu(\cL) \leq \exp(-cn),
\]
where $\mu$ is the spectral measure for $\sig_u\sig_v$ defined in \eqref{eq:mudef}.
\end{theorem}
Thus, the above proves that all but an exponentially small amount of mass is concentrated on sets with size strictly bigger than $\norm{u-v}=n-1$. The following is a stronger quantitative conclusion. 

\begin{theorem}
\label{thm:mainunion}
In the same setup as in Theorem \ref{thm:main}, the following holds. There exists $\eps, c > 0$, independent of $n$, 
such that 
\[
	\sum_{|S| \leq (1 + \eps) n} \mu(S) \leq \exp(-cn).
\]

\end{theorem}

As is apparent, the degree of the improvement over Proposition \ref{thm:chatterjee} is not dramatic.  However the proof strategy is based on certain geometric observations which we hope can be leveraged further in the future as well as be of broader interest. 

Before diving into the proofs let us make a few remarks.  
\begin{remark} The above statements consider the vertices $u = (1, n/2), v = (n, n/2)$ and not any pair of vertices for a reason. Namely, there is a unique shortest path between the points, i.e., the straight line. While, as the proof of Theorem \ref{thm:mainunion} will show, one may indeed take $u$ and $v$ to be vertically separated from each other by $\e n$ for some small constant $\e,$ the current argument does not allow us to take a large vertical separation, such as  $u = (1, n/2), v = (n, n)$, since the number of shortest paths between the two is at least $e^{cn}$ for some universal constant $c>0$ rendering a union bound of an estimate of the form in Theorem \ref{thm:main} useless. Finally, while, for simplicity, we have chosen to present the results in two dimensions, the arguments presented are not planar and work in any dimension yielding the same conclusions as is also the case in \cite{chatterjee2023spin}. 
\end{remark}

\begin{remark}\label{lb1}
Later (see Remark \ref{lowerbound}), we will show that a lower bound (up to a logarithmic term) matching the upper bound in Theorem \ref{thm:main} also holds. Thus, proving a version of Theorem \ref{thm:mainunion} where $(1+\e)$ can be replaced by a growing to infinity sequence $\e_n$ remains a very interesting open problem. Initiating the study of the Fourier spectrum in positive temperature models is another very interesting future direction. 
\end{remark}

\subsection{Acknowledgement}  SG thanks Mahan Mj for the invitation to contribute to the Proceedings of the International Colloquium, Tata Institute of Fundamental Research, 2024, which led to this article and Sourav Chatterjee for introducing him to this topic and helpful discussions. The authors also thank Vilas Winstein for useful comments. SG's research is partially funded by NSF Career grant-1945172.

\section{Proofs}
To set the context for the proofs, let us first review the proof of Proposition \ref{thm:chatterjee} from \cite{chatterjee2023spin}. 
\begin{proof}[Proof of Proposition \ref{thm:chatterjee}]
Consider the effect of fixing the couplings $\{J_e\}$ for all $e \in S \subseteq E$. We
will refer to this collection as $J_S$ for brevity. The conditional expectation of the relative spin is
\begin{align*}
	\E[\sig_u \sig_v \mid J_S] &= \E\Box{\sum_{\k \in \N^E} \al_\k h_\k(\{J_e\}) \cond J_S} \\
							   &= \sum_{\k \in \N^E} \al_\k \cdot \E\Box{h_\k(\{J_e\}) \cond J_S} \\
							   &= \sum_{\k \in \N^E} \al_\k \cdot \E[\prod_{e \in E} h_{k_e}(J_e) \cond J_S].
\end{align*}
However, due to the independence between the couplings of different edges, $\E[h_{k_e}(J_e) \mid J_S] = h_{k_e}(J_e)$ for all $e \in S$
and $= 0$ for all $e \notin S$. Thus, 
\begin{align*}
	\E[\sig_u \sig_v \mid J_S] &= \sum_{\k \sse S} \al_\k \cdot \prod_{e \in S} h_{k_e}(J_e),
\end{align*}
where $\k \sse S$ is used to mean that $k_e = 0$ for all $e \notin S$ (i.e., $E_\k \sse S$). Consequently,
\begin{align}
	\label{eq:sub}
	\E\Box{\E[\sig_u \sig_v \mid J_S]^2} = \sum_{\k \sse S} \al_\k^2,
\end{align}
using the orthonormality of the Hermite polynomials.  Therefore, to show Proposition \ref{thm:chatterjee}, 
it suffices to prove that if $S$ does not connect $u$ and $v$, then $\E[\sig_u \sig_v \mid J_S] = 0$. 

Now we introduce the key observation from \cite{chatterjee2023spin}. {Since $S$ does not connect $u$ and $v$, we may find a \emph{cutset} $C$
separating $u$ and $v$ (a collection of edges $C \sse E$ such that any path from $u$ to $v$ must contain some edge
$e \in C$) with the further property that $C$ is disjoint from $S$}. For instance $C$ could be taken to be the set of boundary edges of the connected component of $u$ induced be the edges in $S$. Consider the configuration of couplings $\{J'_e\}$ obtained by 
flipping the signs of all the edges along $C$, that is:
\[
	J'_e = \begin{cases}
	-J_e & e \in C \\
	J_e & e \notin C
	\end{cases}.
\]
It is not difficult to see that the ground state $\sig'$ for $\{J'_e\}$ can be obtained from $\sig$ simply by flipping the signs of all the spins on ``one side'' of the cutset, that is
\[
	\sig'_w = \begin{cases}
	-\sig_w, & w \in V_u(C) \\
	\sig_w, & w \in V_v(C)
	\end{cases},
\]
where the sides are $V_u(C) = \{w \in V : \text{there is a path from $u$ to $v$ that does not use any edge in $C$}\}$
and $V_v(C)$ (defined similarly). This choice of $\sig'$ preserves the energy along every edge, i.e., 
$J'_e \sig'_u \sig'_v = J_e \sig_u \sig_v$, and
reversing the roles of $J_e$ and $J'_e$ shows that $\sig'$ is indeed optimal.
Note that choosing $V_u(C)$ to be the flipped side was arbitrary, reflecting the aforementioned 
global sign flip symmetry. 

This, in particular, implies that even freezing the disorder on $S,$
the distribution of $\sig_u \sig_v$ is symmetric owing to the symmetry of the variables along $C$, none of which are frozen.
Hence
\begin{align}\label{fliparg}
	\E[\sig_u \sig_v \mid J_S] = 0,
\end{align}
concluding the proof of Proposition \ref{thm:chatterjee}.
\end{proof}

We next start with the proof of Theorem \ref{thm:main} to present the idea in the simplest form.

\begin{proof}[Proof of Theorem \ref{thm:main}]
Following the discussion preceding \eqref{eq:sub}, we will first show 
\begin{align}\label{expcond}
		\Abs{\E[\sig_u\sig_v | J_\cL]} \leq \exp(-cn), \quad \text{with probability at least } 1 - \exp(-cn),
	\end{align}
	for some $c > 0$. 
	This suffices since, owing to $|\sig_u\sig_v|=1,$ the above implies that  
	\begin{equation}\label{fouriermass}
	\E[\Abs{\E[\sig_u\sig_v | J_\cL]}^2]\le e^{-cn}
	\end{equation}
	 for some $c>0.$
		
	We will consider the set of blue edges as in Figure \ref{fig:setup} below.

	\begin{figure}[h]
		\centering
		\begin{tikzpicture}[scale=0.3]
			\tikzset{>=latex}
			\def\sty{densely dotted}
			\foreach \x in {1,...,16} {
				\foreach \y in {1,...,16} {
					\draw[\sty] (\x,\y) -- (\x+1,\y);
					\draw[\sty] (\x,\y) -- (\x,\y+1);
					\draw[\sty] (\x+1,\y) -- (\x+1,\y+1);
					\draw[\sty] (\x,\y+1) -- (\x+1,\y+1);
				}
			}

			\draw[red,very thick] (1,9) -- (17,9);
			\draw[red,fill=red] (1,9) circle (0.1);
			\draw[red,fill=red] (17,9) circle (0.1);
			\node at (1,9) [left] {$u$};
			\node at (17,9) [right] {$v$};

			\foreach \y in {1, ..., 17} {
				\draw[blue,very thick] (6,\y) -- (7,\y);
				\draw[blue, fill=blue] (6,\y) circle (0.1);
				\draw[blue, fill=blue] (7,\y) circle (0.1);
			}

			\node (a) at (20, 16) {$\cL$};
			\draw[->] (a) to (16, 9);

			\node (b) at (-1, 3) {$\es$};
			\draw[->] (b) to (6.5, 9);

		\end{tikzpicture}
		\caption{The setup for Theorem \ref{thm:main}. The blue edges form a ``vertical'' cutset separating $u$ and $v$.}
		\label{fig:setup}
	\end{figure}
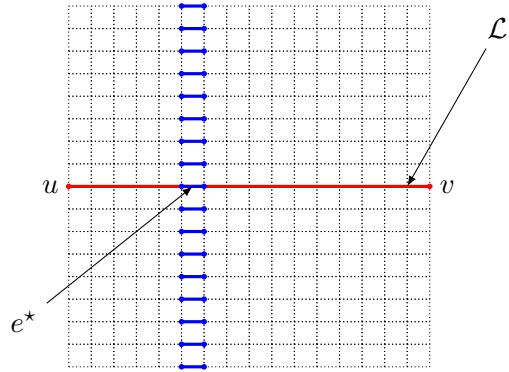

	Since the blue edges form a cut-set separating $u$ and $v$, flipping the signs of the couplings
	along these edges will flip the sign of the relative spin $\sig_u \sig_v$. However, since we condition
	on $J_\cL$, which freezes $J_\es$ for the blue edge $\es$ on $\cL$,  its sign cannot be flipped. This brings us to the following key new observation. With a positive probability (independent of $n$), the edge $\es$ has a ``local barrier'' around it which essentially decouples its status from the relative spin $\sigma_u \sigma_v$ and thereby we are not required to flip it anymore and flipping every other blue edge  suffices.  This will upper bound the conditional expectation by a constant away from $1.$ To make it exponentially small, we will argue that with exponentially small failure probability one can always locate a barrier along $\cL$ and consider a cutset passing through it. 

	\begin{figure}[ht]
		\centering
		\def\largespec{thick, blue}
		\begin{tikzpicture}[scale=0.6]
			\tikzset{>=latex}
			\foreach \x in {1, ..., 3} {
				\draw [dashed, red] (\x, 0) -- (\x + 1, 0);
				\draw [thick, blue] (\x, 1) -- (\x + 1, 1);
				\draw [thick, blue] (\x, -1) -- (\x + 1, -1);
			}
			\foreach \x in {1,..., 4} {
				\draw [dashed, red] (\x, 1) -- (\x, 2);
				\draw [dashed, red] (\x, -1) -- (\x, -2);
			}
			\foreach \y in {-1,...,1} {
				\draw [dashed, red] (0, \y) -- (1, \y);
				\draw [dashed, red] (4, \y) -- (5, \y);
			}

			\draw [thick, blue] (1, -1) -- (1, 1);
			\draw [thick, blue] (4, -1) -- (4, 0);
			\draw [blue, fill=blue] (4, 1) circle (0.08);
			\draw [dashed, red] (4, 0) -- (4, 1);

			\draw [dashed, red] (2, 0) -- (2, 1);
			\draw [dashed, red] (3, 0) -- (3, 1);
			\draw [dashed, red] (2, 0) -- (2, -1);
			\draw [dashed, red] (3, 0) -- (3, -1);

			\draw [very thick, black] (2, 0) -- (3, 0);
			\draw [fill=black] (2, 0) circle (0.08);
			\draw [fill=black] (3, 0) circle (0.08);

			\draw [dotted] (-5, 0) -- (0, 0);
			\draw [dotted] (5, 0) -- (10, 0);

			\node (a) at (2.5, 3) {$e$};
			\draw[->] (a) to (2.5, 0);
			\draw[fill=black] (-5, 0) circle (0.08) node[left] {$u$};
			\draw[fill=black] (10, 0) circle (0.08) node[right] {$v$};
		\end{tikzpicture}
		\caption{The barrier configuration around $e$. The couplings along all 
			the dashed red edges are $\leq 1$ in absolute value, and those along the
		solid blue edges are $> 100$ in absolute value. We refer to the horizontal axis containing $e$ as the \emph{central
	horizontal axis}. The top-right corner which will be a useful reference point in the arguments is made solid.}
		\label{fig:iso}
	\end{figure}
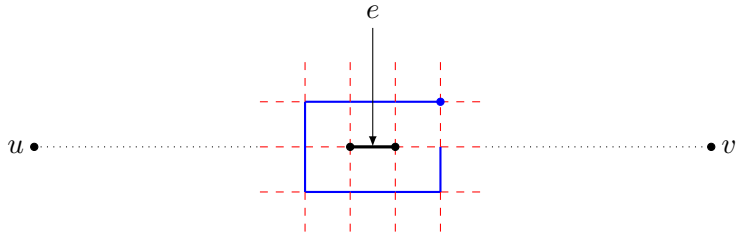
We first specify what a barrier configuration will be for us. Referring to Figure \ref{fig:iso}, the dashed red edges in the figure are assumed to have couplings of absolute value $\leq 1$ (let us call the collection of such edges as $\mathsf{Low}$),
	while the solid blue edges have couplings of absolute value $\geq 100$ (analogously let us denote this collection as $\mathsf{High}$). The key property of this configuration is
	the following. 
	Consider any $e \in \cL$ not too close to $u$ or $v$ (distance of $2$ to each end suffices). If the
	couplings are such that this barrier configuration $\mathsf{Barrier}(e)$ occurs surrounding $e$ (that is, the edges around $e$ satisfy
	the constraints of this configuration), then $\sig_u \sig_v$ is the same as that if the sign of the coupling $J_e$
	were flipped (to $-J_e$). In other words, the relative spin $\sig_u \sig_v$ is \emph{oblivious} to the sign of
	$J_e$ in this configuration. Further, the occurrence of $\mathsf{Barrier}(e)$ is measurable with respect to the 
	the \emph{absolute values}, and hence independent of the  \emph{signs} of the couplings occurring in the description of the event and hence also independent of all the remaining couplings including in particular $J_e$. The next lemma records this formally.

\begin{lemma}\label{barlem} For any $e$, such that the event $\mathsf{Barrier}(e)$ is defined (meaning that the all the edges participating in the event are at distance at least two from $u$ and $v$), its probability is uniformly bounded away from zero. If $\sigma$ is the ground state, then for any configuration as in Figure \ref{fig:iso}, irrespective of the value of $J_e$, it must be that for any $1\le i\le 9$, 
$$\sigma_{w_i}\sigma_{w_{i+1}}J_{(w_i,w_{i+1})}>0.$$
Here $w_1$ denotes the top-right corner as marked in the figure, and $w_{2}, \ldots, w_{10}$ denotes the remaining boundary vertices  traversed in the anti-clockwise direction. 
Thus, this determines $\sigma_{w_i}$ for all $1\le i\le 10$ and in particular $\sigma_u\sigma_v$ is independent of $J_e.$ 	Consequently, if one flips the sign of $J_e$, $\sig_u\sig_v$ remains unchanged.
\end{lemma}
\begin{proof} Since the first claim is trivial owing to Gaussians being fully supported we focus on the latter claims. 

First note that owing to the global flip symmetry, we can assume $\sigma_{w_1}=1.$ Let us assume that $\sigma$ does not satisfy the  conclusion of the lemma and arrive at a contradiction. Let $\hat \sigma$ be the spin configuration which agrees with $\sigma$ off the vertices $w_i$ where $1\le i\le 10$ and define  $\hat \sigma_{w_i}$ inductively for $i$ going from $2$ to $10$ by the relation $\sigma_{w_{i-1}}\sigma_{w_{i}}= {\mathsf{sign}}\,{J_{(w_i,w_{i+1})}}$.
We will now show that $H(\hat \sigma)> H(\sigma)$ which finishes the proof of this claim. Since, by hypothesis, there exists at least one $1\le i\le 9$ such that $\sigma_{w_{i-1}}\sigma_{w_{i}}J_{(w_i,w_{i+1})}<0$, it follows that 
\begin{equation}\label{improvement}
H(\hat \sigma)-H(\sigma)\ge 2\min_{1\le i\le 9}|J_{(w_i,w_{i+1})}|-2\sum_{e\in {\mathsf{Low} }}|J_e|\ge 200-40=160. 
\end{equation}
The proof of the final claim is a consequence of the domain Markov property, i.e.,  fixing $\sigma_{w_{i}}$ for $1\le i\le 10$, decouples $\sigma$ outside the box and inside the box. Thus it follows that $\sigma_{u}$ and $\sigma_v$ (note that $\sigma_{w_1}$ is pinned to be $1$) are independent of $J_e$ which finishes the proof. 
\end{proof}

We now proceed with the remainder of the proof. 
	Denote by $|J|$ the collection of absolute values $|J_f|$ for all edges $f$ (not just those on $\cL$). 
Let us call the set of all such $|J| $ which have at least one  edge $e \in \cL$ such that the event $\mathsf{Barrier}(e)$ occurs as $\mathsf{Good}.$ Since $\mathsf{Barrier}(e)$ occurs with constant probability and there are linear in $n$ many such events supported on disjoint edges, and thereby independent, any of whose occurrences imply $\mathsf{Good}$ we have
\begin{equation}\label{expbound}
\P(\mathsf{Good}) \ge 1 - \exp(-cn),
\end{equation}
for some $c>0.$
	
	The upcoming Lemma \ref{condsym} shows 
	\begin{equation}\label{precdisp}
	\E[\sig_u\sig_v \cond J_\cL, |J| \,\,\,\text{such that}\,\,\, |J| \in \mathsf{Good}] = 0.
	\end{equation}  
This implies \eqref{expcond} and by \eqref{fouriermass} finishes the proof. 
To see why \eqref{expcond} follows from the above display,	
note that
	\begin{align}
	&\E[\sig_u\sig_v \cond J_\cL] =\\
	 &\P\left( |J| \in \mathsf{Good}\cond J_\cL\right)\E[\sig_u\sig_v \cond J_\cL, |J| \in \mathsf{Good}] + \P\left( |J| \notin \mathsf{Good}\cond J_\cL\right)\E[\sig_u\sig_v \cond J_\cL, |J| \notin \mathsf{Good}].
	\end{align}
	
	 The first term is zero by \eqref{precdisp} and the second term is bounded in absolute value by $\P(\mathsf{Good}^c \mid J_\cL )$. Markov's inequality and  \eqref{expbound} imply that $\P(\mathsf{Good}^c \mid J_\cL ) \le e^{-cn}$ with probability at least $1-e^{-cn}$ for some absolute constant $c>0$ which proves \eqref{expcond}.

Finally, \eqref{eq:sub} implies that $$\sum_{S \sse \cL}\mu(S)\le e^{-cn}$$ for some $c>0.$ Since any $S$ which is a strict subset of $\cL$ fails to connect $u$ and $v$,  by Proposition \ref{thm:chatterjee}, $$\sum_{S \sse \cL}\mu(S)=\mu(\cL).$$ This finishes the proof. 
	 \end{proof}

	\begin{lemma}\label{condsym}
	\begin{align}
		\label{eq:toshow}
		\E[\sig_u\sig_v \cond J_\cL, |J| \,\,\,\text{such that}\,\,\, |J| \in \mathsf{Good}] = 0.
	\end{align}
	\end{lemma}
	
	\begin{proof}
Suppose $\es$ is an edge on $\cL$ such that $\mathsf{Barrier}(\es)$ has occurred. Let us now reiterate the randomness that has been revealed due to the above conditioning. This includes $J_\cL$ and $|J_e|$ for all $e \notin \cL.$ Thus the variables $\{\mathsf{sign}(J_e): e\notin \cL\}$ are still independent uniform $\{\pm 1\}$ random variables. In particular, among the vertical cut-set $C$ of blue edges going through $\es$, as shown in Figure \ref{fig:setup}, while the absolute values of all of them are revealed, only the sign of $J_\es$ is frozen. However, by Lemma \ref{barlem}, on $\mathsf{Barrier}(\es),$ the sign of $\sig_u\sig_v$ does not depend on the value and hence the sign of $J_\es$.
This allows us to implement the cut-set flipping argument outlined before \eqref{fliparg} verbatim. More precisely, let $\hat J$ be the Gaussian configuration where 
\[
	\hat J_e = \begin{cases}
	-J_e & e \in C, e \neq \es \\
	J_e & e\,\,\, \text{otherwise}.
	\end{cases}.
\]
and $J'$ is the same as $\hat J$ except the sign of $\es$ is flipped too. Since $\mathsf{Barrier}(\es)$ occurs, $\sig_u\sig_v$ for $J'$ and $\hat J$ are the same. Further, by the argument  before \eqref{fliparg}, $\sig_u\sig_v$ for $J$ and $J'$ are of opposite signs and hence the same holds for $J$ and $\hat J$. 
Since the signs of all the edges $e \in C, e \neq \es$ conditioned on $\{J_\cL, |J|, |J| \in \mathsf{Good}\}$ are i.i.d. uniform $\{\pm1\}$, it follows that $\sig_u\sig_v$ is even conditionally uniformly distributed on $\{\pm1\}$ which finishes the proof. 
	\end{proof}
	
\begin{remark}\label{lowerbound} In this remark, as indicated in Remark \ref{lb1}, we show how one may also obtain a lower bound counterpart to Theorem \ref{thm:main}.
Consider the set $\widehat \cL$ given by $\cL$ and all the edges which share a vertex with one of the edges in $\cL.$ 
An argument similar to \eqref{improvement} can be employed to prove that $|\E(\sigma_{u}\sigma_v \mid J_{\widehat \cL})|=1$ with probability at least $e^{-cn\log n}$.  Consider the event that $J_{e}>100$ for all $e\in \cL$ and $|J_{e}|<\frac{1}{n}$ for all $e\in \widehat \cL\setminus \cL$. On this event, for all the vertices $w$ including $v,$ incident on edges $e\in \cL$, it must be must be the case, provided that $\sigma_u$ is pinned to be $1,$ that $\sigma_w=1$ and hence in particular $\sigma_u\sigma_v=1.$ To see this, note that any configuration $\sigma$ such that $\sigma_u=1$ but $\sigma_w\neq 1 $ for some $w$ as above has strictly smaller energy than $\hat \sigma$ which agrees with $\sigma$ off $\cL$ but is equal to $1$ for all such $w$. The difference in energy being at least $100- 2\cdot n\cdot \frac{1}{n}$. The $100$ term appears because of the energy change from edges in $\cL$ whereas the $2\cdot n\cdot \frac{1}{n}$ term is due to the edges adjacent to $\cL$, (there are at most $2n$ such edges in two dimensions and the constants have to be modified accordingly in higher dimensions).  Now the probability of the above event is at least $e^{-cn\log n}$ and hence this, by \eqref{eq:sub}, implies that \begin{align}\label{expcond34534}
		\sum_{\cL \subset S\subset \widehat \cL}\mu (S)\ge e^{-cn\log n},
	\end{align}
	for some $c > 0$. Since the number of such subsets $S$ is $e^{cn}$, it follows that there must be a set $S$ such that $\mu(S)\ge e^{-cn\log n}$ and $\cL \subset S \subset \widehat \cL$. Above the constant $c$ is always positive but its value is allowed to change from line to line. 
\end{remark}

\begin{remark}\label{higdimbar}The reader may notice that while we had claimed that our arguments are not planar, it might not be completely obvious what the analogous barrier construction is in general dimensions. It turns out that the right generalization is the following. Consider a box centered at $e$ of side lengths, say $3$ in the direction of $e$ and $2$ in the other $d-1$ directions. Now take a spanning tree of the surface of the box (which is connected) rooted at the one of the corners and force the edges of the tree to have high values and hence color them blue, whereas every other edge adjacent to the blue edges are red edges and have low values. This by an argument as in \eqref{improvement} allows one to conclude that for the ground state $\sigma,$ the spins $\sigma_u$ of all the vertices on the surface of the box are determined iteratively by exploring the spanning tree from the root ensuring that $\sigma_u\sigma_v J_{(u,v)}>0$ for every edge $(u,v)$ in the tree.  
\end{remark}

	We will next prove Theorem \ref{thm:mainunion} by upgrading the proof of Theorem \ref{thm:main}. This is based on the observation  that the proof of Theorem \ref{thm:main} does not truly require the conditioned set $S$
	to be $\cL$ and small deviations from it may be tolerated. The sets connecting $u$ and $v$ having size at most $(1+\e)n$ must have long stretches of horizontal segments allowing us to implement an argument like the one for Theorem \ref{thm:main}  as well have small entropy or total count permitting a union bound to conclude that the total spectral mass across all of them is small.

\def\st{\mathsf{straight}}
	\def\env{\mathsf{envelope}}
\begin{proof}[Proof of Theorem \ref{thm:mainunion}]
	The theorem is an immediate consequence of the following lemma by a union bound and observing that $\e\log(\frac{1+\e}{\e})$ can be made arbitrarily small by choosing $\e$ small enough. 
\end{proof}	
	
	  \newpage
	
	\begin{lemma}\label{key}There exists $c,c_2>0$ such that for all small $\e>0$ there exists a collection $\mathscr{C}$ of at most $\Exp{n
																				   \cdot \eps \cdot \Rnd{\log \F{1 
																					 }{\eps}
																			   + c_2}} $ sets $C$ such that the following holds.
																			 
												$1.$ For any $S$ of size at most $(1+\e)n$ which connects $u$ and $v$ there exist a $C\in \mathscr{C}$ such that $S\subset C.$ 	\\\\
$2.$ $\displaystyle{
	\sum_{S \subset C} \mu(S) \leq \exp(-cn),}$				\end{lemma}
	
	The lemma has two parts. The first part is a combinatorial or counting statement. The second part essentially follows by the argument for Theorem \ref{thm:main}.  We begin with the first part. 
	\begin{proof}[Proof of 1.]
	Consider any $S$ as in the statement of the lemma. 
	Let $W = 5$ (the width of the barrier in
	Figure \ref{fig:iso}),
	and suppose, for simplicity, that $W$ divides $n$ (a simple modification of this argument will work otherwise).
	Partition the grid into $n/W$ columns of width $W$, and let $I_j$ be the $j$-th column, i.e. the set of all edges whose at least one endpoint has $x$-coordinate in the semi-closed interval $[jW,(j+1)W)$. Note that since $S$ is a subset of edges that 
	connects $u$ and $v$, we must have that $|I_j \cap S|$ contains a connected set of size at least $W$ for each $j$. However, for small enough $\e,$ on account of the constraint $|S| \leq (1 + \eps) n$, there also must be linearly many of the columns $I_j$, say at least $50\%$ of them, where $|I_j \cap S| = W$, that is, $I_j
	\cap S$ is a ``straight sequence of edges'' of length $W$ (We call such columns ``straight'', and call their collection $\st(S)$. Here $j$ and $I_j$ may be identified and consequently we may say either $j  \in
	\st(S)$ or $I_j \in \st(S)$; see Figure \ref{fig:union1} for an illustration.).
	To see this suppose that only  $\t$-fraction of
	these columns are straight. Then, the remaining $(1 - \t)$ fraction of them satisfy $|I_j \cap 
	S| \geq W + 1$, resulting in a total edge-count of at least 
	\[
		\t \cdot \f{n}{W} \cdot W + (1 - \t) \f{n}{W} \cdot (W + 1) = n \cdot \Rnd{1 + \f{1 - \t}{W}}.
	\]
	This must be $\leq (1 + \eps) n$, so we must have
	\[
		\t \geq 1 - \eps W.
	\]
We will now classify the different $S$s according to their straight parts, that is, two sets $S,
	S'$ fall in the same class $\cG$ if $\st(S) = \st(S')$, and for each $j \in
	\st(S)$, $I_j \cap S = I_j \cap S'$. We will identify $\cG$ with the
	collection $J$ of indices $j$ such that $I_j$ is straight, and the collection $Y = \{y_j : j \in J\}$ of the
	$y$-coordinates of the edges in $I_j \cap S$.

	\begin{figure}[ht]
		\centering
		\begin{tikzpicture}[scale=0.4]
			\tikzset{>=latex}

			\foreach \y in {0,...,6} {
				\draw[blue, dashed] (\y*3, 0) -- (\y*3, 18);
			}

			\draw[dotted, pattern=north west lines, pattern color=orange] (6, 0) rectangle ++(3, 18);
			\draw[dotted, pattern=north west lines, pattern color=orange] (12, 0) rectangle ++(3, 18);

			\draw[red,fill=red] (0,9) circle (0.1);
			\draw[red,fill=red] (18,9) circle (0.1);
			\node at (0,9) [left] {$u$};
			\node at (18,9) [right] {$v$};

			\draw[red,very thick] (0, 9) -- (2, 9) -- (2, 10) -- (6, 10) -- (6, 8) --
			(6, 7) -- (5, 7) -- (5, 6) -- (10, 6) -- (10, 10) -- (11, 10) --
			(16, 10) -- (16, 9) -- (18, 9);

			\draw[red,very thick] (1, 9) -- (1,11) -- (3, 11);
			\draw[red,very thick] (16, 7) -- (17, 7) -- (17, 6);

			\foreach \x in {1,...,6} {
				\node at (\x*3 - 1.5, 5) [below] {$I_\x$};
			}

			\draw[dotted] (16, 10) -- (18, 10) node[right] {$y_5$};
			\draw[dotted] (0, 6) node[left] {$y_3$} -- (5, 6);
		\end{tikzpicture}
		\caption{Quantities involved in the proof of  Lemma \ref{key}. The solid red edges form the set $S$.
			The columns $I_3$ and $I_5$ (shaded in orange) are ``straight'' (note that $I_3$ is straight even though its
			left-boundary has some edges in $S$). The group $\cG$ for this set is defined by $J = \{3, 5\}$ and $Y = \{y_3,
			y_5\}$ as shown in the figure. Note that $S$ is not stipulated to itself be connected and hence may contain additional edges outside the path connecting $u$ and
			$v$.}
		\label{fig:union1}
	\end{figure}
Since there are at least $(1 - \eps W) \cdot \f{n}{W}$ straight
	segments out of a total of $\f{n}{W}$, the number of ways to choose $J$ is at most
	\begin{align}
		\sum_{m \geq (1 - \eps W) \cdot n/W} \binom{n/W}{m}  \le  \Exp{n \cdot \eps \cdot \Rnd{\log \eps^{-1} + c_1}}. \label{eq:choosej}
	\end{align}
The above bound follows from the standard Chernoff bound for Binomials and standard asymptotics (as $\e \to 0$) of the entropy functional appearing in the Chernoff bound. 

	Now let us fix the set $J$ and consider the number of ways of choosing $Y$. Note that any path connecting $u$ and $v$ 
	in $S$ must contain every straight segment $I_j \cap S$ for $j \in J$. Thus a crude upper bound on the number of ways to choose $Y$ is the number of ways to connect $u$ and $v$ using at most
	$(1 + \eps) n$ edges since, given the set $J$, such a path exactly recovers the straight segments we are seeking. 
	
	The path counting problem is straightforward. There are at most $(1+\e)n$ steps. At least $n$ of them must be of the form $\{(x, y), (x + 1, y)\}.$ The remaining edges can be of the three remaining types, i.e,., going left, north or south. 

	 Therefore, the overall count is at most
	\begin{align}
		\sum_{m \leq (1 + \eps) n} \sum_{r \leq \eps n} \binom{m}{r} \cdot 3^r &\leq \eps(1 + \eps) n^2 \cdot \binom{(1
		+ \eps) n}{\eps n} \cdot 3^{\eps n} \\
																			   &\leq \eps(1 + \eps) n^2 \cdot \Exp{n
																				   \cdot \eps \cdot \Rnd{\log \F{1 +
																					   \eps}{\eps}
																			   + c_2}} \label{eq:choosey},
	\end{align}
	where $c_2$ is another absolute constant. Putting together equations \eqref{eq:choosej} and \eqref{eq:choosey}, 
	the total number of choices of $\cG$ is at
	most
	\begin{equation}
		\label{eq:chooseg}  \Exp{3 n \cdot \eps \cdot \Rnd{\log \eps^{-1} + c_3}},
	\end{equation}
	for all $\eps$ small enough, where $c_3$ is universal. 

	\begin{figure}[ht]
		\centering
		\begin{tikzpicture}[scale=0.4]
			\tikzset{>=latex}

			\foreach \y in {0,...,6} {
				\draw[blue, dashed] (\y*3, 0) -- (\y*3, 18);
			}

			\draw[red,fill=red] (0,9) circle (0.1);
			\draw[red,fill=red] (18,9) circle (0.1);
			\node at (0,9) [left] {$u$};
			\node at (18,9) [right] {$v$};

			\foreach \x in {1,...,6} {
				\node at (\x*3 - 1.5, 5.2) [below] {$I_\x$};
			}

			\draw[red, very thick] (6, 6) -- (9, 6);
			\draw[red, very thick] (12, 10) -- (15, 10);

			\foreach \j in {0, 1, 3, 5} {
				\foreach[evaluate] \x in {\j*3, \j*3 + 1, \j*3 + 2} {
					\foreach \y in {0,...,17} {
						\draw[red, very thick] (\x, \y) -- (\x + 1, \y);
						\draw[red, very thick] (\x, \y) -- (\x, \y+1);
						\draw[red, very thick] (\x, \y + 1) -- (\x + 1, \y + 1);
					}
				}
				\draw[red, very thick] (\j*3, 0) -- (\j*3, 18);
				\draw[red, very thick] (\j*3 + 3, 0) -- (\j*3 + 3, 18);
			}

			\draw[dotted] (15, 10) -- (18, 10) node[right] {$y_5$};
			\draw[dotted] (0, 6) node[left] {$y_3$} -- (6, 6);
		\end{tikzpicture}
		\caption{Illustration of $S_{\cG}$ where $\cG$ corresponds to the one in Figure \ref{fig:union1}. The edges in  $S_{\cG}$ are shown in red.
		Note that the boundaries of the columns $I_3$ and $I_5$ are also included in $S_{\cG}$.}
		\label{fig:union2}
	\end{figure}
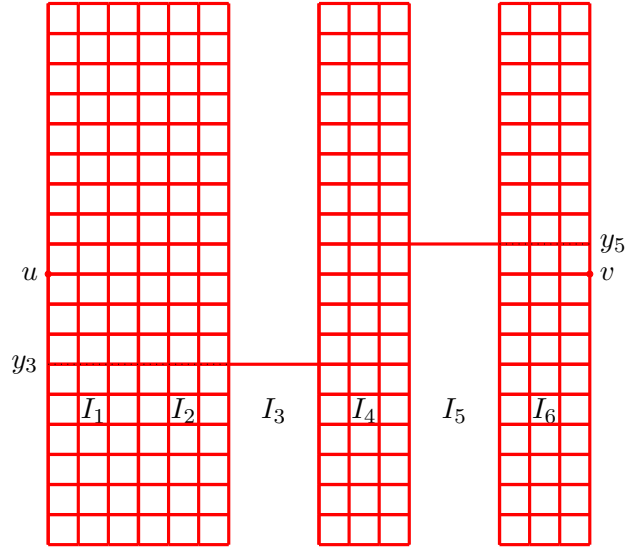
For each $\cG,$ consider the set $S_{\cG}$ given by the union of the straight segments present in any set $S \in \cG$ (this is well defined by the definition of $\cG$) and all the edges present in every remaining column as well as all the vertical edges on the boundaries of the straight columns. See for instance Figure \ref{fig:union2}. 
Since for any $S\in \cG$  we have $S \subset S_\cG,$ taking $C=S_{\cG}$ and the collection $\mathscr{C}$ to be $\{S_{\cG}\}_{\cG}$ finishes the proof of the first part.
\end{proof}

	\begin{proof}[Proof of 2.]
	The argument now proceeds in the same way as for Theorem
	\ref{thm:main} and it suffices to show that 	for any $\cG$ as in the proof of 1.,  with probability at least $1 - \Exp{-c_5 n}$,
		\[
		\Abs{\E\Box{\sig_u \sig_v \cond J_{S_\cG}}} \leq \Exp{-c_5 n},
	\]
	for some constant $c_5>0$ for all small enough $\e>0.$

 By choosing $\eps < \f{1}{2W}$, one may ensure that at least half of
	the columns $I_j$ are straight. Each such segment is the central horizontal axis of a barrier configuration with constant probability independently across each segment. On the event that a barrier occurs (let us call it $\sf{Good}$ as before), conditioning on $J_{S_{\cG}}$ and $|J|$ allows flipping  the signs of the variables along a vertical cutset passing through the middle of the straight segment witnessing the barrier configuration, showing that $$\E[\sig_u \sig_v \cond J_{S_{\cG}}, |J|, \sf{Good}] = 0.$$
	This finishes the proof.  
\end{proof}

\begin{remark}\label{highenv}
We conclude by remarking that the higher dimensional version of the above argument is essentially the same where instead of columns we consider slabs of width $W$, i.e., the set of all points whose $x$ coordinate lands in an interval of the type $[jW, (j+1)W).$
\end{remark}

\appendix

\bibliographystyle{alpha}
\bibliography{ref}

\end{document}

%% file: preamble.tex
\usepackage[letterpaper, hmargin=1.0in, vmargin=1.0in]{geometry}
\usepackage{import}
\usepackage{xifthen}
\usepackage{pdfpages}
\usepackage{transparent}
\usepackage{cancel}
\usepackage[normalem]{ulem}

\usepackage{tikz}
\usetikzlibrary{arrows.meta}
\usepackage{float}
\usepackage{mathtools, amsmath, amsfonts, amssymb, amsthm, mathrsfs}
\usepackage{pgffor}
\usepackage{etoolbox}
\usepackage[shortlabels]{enumitem}
\usepackage{xcolor}
\definecolor{CustomBlue}{RGB}{23, 86, 118}
\usepackage[unicode=true, colorlinks=true, linkcolor=CustomBlue, citecolor=CustomBlue]{hyperref}
\hypersetup{pdfauthor={Name}}
\usepackage{setspace}

\usepackage[T1]{fontenc}
\newcommand{\sse}{\subseteq}

\newcommand{\N}{\mathbb{N}}

\newcommand{\Z}{\mathbb{Z}}
\newcommand{\eps}{\varepsilon}

\renewcommand{\t}{\theta}
\newcommand{\sig}{\sigma}

\newcommand{\al}{\alpha}
\newcommand{\be}{\beta}

\newcommand{\Nor}[1]{\mathcal{N}\left(#1\right)}
\newcommand{\Exp}[1]{\exp\left(#1\right)}
\newcommand{\Abs}[1]{\left|#1\right|}

\renewcommand{\Box}[1]{\left[#1\right]}
\newcommand{\Rnd}[1]{\left(#1\right)}

\renewcommand{\bar}[1]{\overline{#1}}

\newcommand{\argmax}{\text{argmax}}
\makeatletter
\newcommand{\E}[1][\@nil]{%
	\def\tmp{#1}%
	\ifx\tmp\@nnil
		\mathbb{E}
	\else
		\mathbb{E}\left[#1\right]
\fi}

\renewcommand{\P}[1][\@nil]{%
	\def\tmp{#1}%
	\ifx\tmp\@nnil
		\mathbb{P}
	\else
		\mathbb{P}\Rnd{#1}
\fi}
\makeatother

\newcommand{\Cov}{\mathrm{Cov}}
\newcommand{\Var}{\mathrm{Var}}
\renewcommand{\hat}[1]{\widehat{#1}}
\newcommand{\norm}[1]{\left\lVert#1\right\rVert}
\newcommand{\mnorm}[1]{{\left\vert\kern-0.25ex\left\vert\kern-0.25ex\left\vert #1 
\right\vert\kern-0.25ex\right\vert\kern-0.25ex\right\vert}}

\newcommand{\cG}{\mathcal{G}}

\newcommand{\f}[2]{\frac{#1}{#2}}
\newcommand{\F}[2]{{\left(\frac{#1}{#2}\right)}}

\newcommand{\x}{\bar{x}}
\newcommand{\y}{\bar{y}}

\usepackage{mathtools}
\mathtoolsset{showonlyrefs}

\newtheorem{theorem}{Theorem}[section]
\newtheorem{lemma}[theorem]{Lemma}
\newtheorem{proposition}[theorem]{Proposition}

\newtheorem{definition}[theorem]{Definition}
\theoremstyle{definition}
\newtheorem{remark}[theorem]{Remark}

\numberwithin{theorem}{section}

\numberwithin{equation}{section}
\numberwithin{figure}{section}

\newcommand{\cL}{\ensuremath{\mathcal L}}

\def\({\left(}
\def\){\right)}

\newcommand{\e}{\varepsilon}

%% file: EA.bbl
\begin{thebibliography}{ANS19}

\bibitem[AC16]{auf}
Antonio Auffinger and Wei-Kuo Chen.
\newblock Universality of chaos and ultrametricity in mixed p-spin models.
\newblock {\em Communications on Pure and Applied Mathematics},
  69(11):2107--2130, 2016.

\bibitem[AH20]{arguin}
Louis-Pierre Arguin and Jack Hanson.
\newblock On absence of disorder chaos for spin glasses on z\^{}d.
\newblock 2020.

\bibitem[ANS19]{ans}
L-P Arguin, CM~Newman, and DL~Stein.
\newblock A relation between disorder chaos and incongruent states in spin
  glasses on z\^{} d z d.
\newblock {\em Communications in Mathematical Physics}, 367:1019--1043, 2019.

\bibitem[BKS99]{bks}
Itai Benjamini, Gil Kalai, and Oded Schramm.
\newblock Noise sensitivity of boolean functions and applications to
  percolation.
\newblock {\em Publications Math{\'e}matiques de l'Institut des Hautes
  {\'E}tudes Scientifiques}, 90:5--43, 1999.

\bibitem[BM85]{bm1}
AJ~Bray and MA~Moore.
\newblock Critical behavior of the three-dimensional ising spin glass.
\newblock {\em Physical Review B}, 31(1):631, 1985.

\bibitem[BM87]{bm2}
Alan~J Bray and Michael~A Moore.
\newblock Chaotic nature of the spin-glass phase.
\newblock {\em Physical review letters}, 58(1):57, 1987.

\bibitem[Cha09]{chachaos}
Sourav Chatterjee.
\newblock Disorder chaos and multiple valleys in spin glasses.
\newblock {\em arXiv preprint arXiv:0907.3381}, 2009.

\bibitem[Cha14]{sc}
Sourav Chatterjee.
\newblock {\em Superconcentration and related topics}, volume~15.
\newblock Springer, 2014.

\bibitem[Cha23]{chatterjee2023spin}
Sourav Chatterjee.
\newblock Spin glass phase at zero temperature in the edwards-anderson model.
\newblock {\em arXiv preprint arXiv:2301.04112}, 2023.

\bibitem[Che13]{chen1}
Wei-Kuo Chen.
\newblock Disorder chaos in the sherrington--kirkpatrick model with external
  field.
\newblock 2013.

\bibitem[Che14]{chen2}
Wei-Kuo Chen.
\newblock Chaos in the mixed even-spin models.
\newblock {\em Communications in Mathematical Physics}, 328:867--901, 2014.

\bibitem[CP18]{chen3}
Wei-Kuo Chen and Dmitry Panchenko.
\newblock Disorder chaos in some diluted spin glass models.
\newblock 2018.

\bibitem[EA75]{ea}
Samuel~Frederick Edwards and Phil~W Anderson.
\newblock Theory of spin glasses.
\newblock {\em Journal of Physics F: Metal Physics}, 5(5):965, 1975.

\bibitem[Eld20]{eldan}
Ronen Eldan.
\newblock A simple approach to chaos for p-spin models.
\newblock {\em Journal of Statistical Physics}, 181(4):1266--1276, 2020.

\bibitem[FH86]{fh}
Daniel~S Fisher and David~A Huse.
\newblock Ordered phase of short-range ising spin-glasses.
\newblock {\em Physical review letters}, 56(15):1601, 1986.

\bibitem[GPS10]{gps}
Christophe Garban, G{\'a}bor Pete, and Oded Schramm.
\newblock The fourier spectrum of critical percolation.
\newblock 2010.

\bibitem[NS03]{ns}
Charles~M Newman and Daniel~L Stein.
\newblock Ordering and broken symmetry in short-ranged spin glasses.
\newblock {\em Journal of Physics: Condensed Matter}, 15(32):R1319, 2003.

\bibitem[Pan13]{pan}
Dmitry Panchenko.
\newblock {\em The sherrington-kirkpatrick model}.
\newblock Springer Science \& Business Media, 2013.

\bibitem[SK75]{sk}
David Sherrington and Scott Kirkpatrick.
\newblock Solvable model of a spin-glass.
\newblock {\em Physical review letters}, 35(26):1792, 1975.

\end{thebibliography}
